\documentclass[reqno]{amsart}

\usepackage{enumitem}
\usepackage{amssymb}
\usepackage{hyperref}

\newcommand*{\mailto}[1]{\href{mailto:#1}{\nolinkurl{#1}}}
\newcommand{\arxiv}[1]{\href{http://arxiv.org/abs/#1}{arXiv:#1}}

\newtheorem{theorem}{Theorem}[section]
\newtheorem{definition}[theorem]{Definition}
\newtheorem{lemma}[theorem]{Lemma}

\newtheorem{proposition}[theorem]{Proposition}
\newtheorem{corollary}[theorem]{Corollary}

\newtheorem{remark}[theorem]{Remark}

\newcommand{\R}{{\mathbb R}}

\newcommand{\Z}{{\mathbb Z}}
\newcommand{\C}{{\mathbb C}}
\newcommand{\D}{{\mathcal D}}

\newcommand{\cR}{{\mathcal R}}

\newcommand{\bl}{{\bf l}}

 \newcommand{\br}{|\kern-.25em|\kern-.25em|}
 \newcommand{\floor}[1]{\lfloor#1 \rfloor}
\newcommand{\ceil}[1]{\lceil#1 \rceil}

\newcommand{\I}{\mathrm{i}}
\newcommand{\E}{\mathrm{e}}
\DeclareMathOperator{\re}{Re}
\DeclareMathOperator{\im}{Im}

\newcommand{\ve}{\varepsilon}

\numberwithin{equation}{section}

\begin{document}

\title[Dispersion Estimates for Discrete Dirac Equations]{Dispersion Estimates for One-Dimensional Discrete Dirac Equations}

\author[E.\ Kopylova]{Elena Kopylova}
\address{Faculty of Mathematics\\ University of Vienna\\
Oskar-Morgenstern-Platz 1\\ 1090 Wien\\ Austria\\ and Institute for Information Transmission Problems\\ Russian Academy of Sciences}
\email{\mailto{Elena.Kopylova@univie.ac.at}}
\urladdr{\url{http://www.mat.univie.ac.at/~ek/}}

\author[G.\ Teschl]{Gerald Teschl}
\address{Faculty of Mathematics\\ University of Vienna\\
Oskar-Morgenstern-Platz 1\\ 1090 Wien\\ Austria\\ and International Erwin Schr\"odinger
Institute for Mathematical Physics\\ Boltzmanngasse 9\\ 1090 Wien\\ Austria}
\email{\mailto{Gerald.Teschl@univie.ac.at}}
\urladdr{\url{http://www.mat.univie.ac.at/~gerald/}}

\thanks{J. Math. Anal. Appl. {\bf 434}, 191--208 (2016)}
\thanks{{\it Research supported by the Austrian Science Fund (FWF) under Grant No.\ P27492-N25.}}


\begin{abstract}
We derive dispersion estimates for solutions of the one-dimensional discrete perturbed Dirac equation.
To this end,  we develop basic scattering theory and establish a limiting absorption principle for
discrete perturbed Dirac operators.
\end{abstract}

\maketitle

\section{Introduction}

In the present paper, we are concerned with the one-dimensional discrete Dirac equation
\begin{equation} \label{Dir}
 \I\dot{\bf w}(t):=\D{\bf w}(t)=(\D_0+Q){\bf w}(t),
\quad {\bf w}_n=(u_n,v_n)\in\C^2,\quad n\in\Z,
\end{equation}
where the unperturbed discrete Dirac operator $\D_0$ is given by
\[
\D_0=\begin{pmatrix}
m   & d\\
d^* & -m
\end{pmatrix},\quad m>0,
\]
with $(du)_n=u_n-u_{n+1}$.
We suppose that the matrix potential $Q=(q^{ij})_{i,j=1,2}$ is real-valued
and satisfies the conditions:
\begin{equation}\label{q-con}
q_n^{12} \equiv q_n^{21} \quad
\text{and}\quad q_n^{21}\not =-1, \quad \forall n\in\Z.
\end{equation}
To formulate our results, we introduce the weighted spaces 
$\ell^p_{\sigma}=\ell^p_{\sigma}(\Z)$,
$\sigma\in\R$,  associated with the norm
\begin{equation*}
\Vert u\Vert_{\ell^p_{\sigma}}= \begin{cases} 
\left( \sum_{n\in\Z} (1+|n|)^{p\sigma} |u(n)|^p\right)^{1/p}, & \quad p\in[1,\infty),\\
\sup_{n\in\Z} (1+|n|)^{\sigma} |u(n)|, & \quad p=\infty, \end{cases}
\end{equation*}
and the case $\sigma=0$ corresponds to the usual $\ell^p_0=\ell^p$ spaces without weight.
We also set $\bl^p_{\sigma}=\ell^p_{\sigma}\oplus \ell^p_{\sigma}$.

Under the assumption $q^{ij}\in\ell^1_2$ the spectrum of $\D$ consists of a purely absolutely continuous part
covering $\overline\Gamma$, where $\Gamma=(-\sqrt{m^2+4},-m)\cup (m,\sqrt{m^2+4})$, 
plus a finite number of eigenvalues located in $\R\setminus\overline\Gamma$.
In addition, there could be resonances at the boundaries of the continuous spectrum, that is, there could
be a corresponding bounded solution of the underlying difference equation at these energies. 

As our first main result, we will prove the following $\bl^1\to \bl^\infty$ decay
\begin{equation}\label{fullp}
\Vert \E^{-\I t\D}P_{c}\Vert_{\bl^1\to \bl^\infty}=\mathcal{O}(t^{-1/3}),\quad t\to\infty,
\end{equation}
under the assumptions $q^{ij}\in\ell^1_2$  
in the non-resonant case and $q^{ij}\in\ell^1_3$ in the resonant case.
Here $P_{c}$ is the orthogonal projection in $\bl^2$ onto the
continuous spectrum of $\D$.

In combination with conservation of the $\bl^2$ norm
\eqref{fullp} also gives rise to the usual interpolation and Strichartz estimates.
Moreover, \eqref{fullp} implies
\[
\Vert \E^{-\I t\D}P_{c}\Vert_{\bl^2_\sigma\to \bl^2_{-\sigma}}=\mathcal{O}(t^{-1/3}),\quad t\to\infty, \quad \sigma>1/2.
\]
However, we will in fact establish the stronger result
\begin{equation}\label{as1-new}
\Vert \E^{-\I t\D} P_c\Vert_{\bl^2_\sigma\to \bl^2_{-\sigma}}=\mathcal{O}(t^{-1/2}),\quad t\to\infty,\quad\sigma>1/2.
\end{equation}
For the remaining results we restrict ourselves to the case when
the edges of the spectrum $\omega=\pm m,\pm\sqrt{4+m^2}$ are
no resonances for $\D$. Then for $q^{ij}\in\ell^1_3$, $i,j=0,1$, we show that
\begin{equation}\label{as-new}
\Vert \E^{-\I t\D}P_c\Vert_{\bl^1_{3/2}\to \bl^\infty_{-3/2}}=\mathcal{O}(t^{-4/3}),
\quad t\to\infty.
\end{equation}
Moreover, for $q^{ij}\in\ell^1_3$ in  the non-resonant case, we prove
\begin{equation}\label{as2-new}
\Vert \E^{-\I t\D} P_c\Vert_{\bl^2_\sigma\to \bl^2_{-\sigma}}=\mathcal{O}(t^{-3/2}),\quad t\to\infty,\quad\sigma>2.
\end{equation} 

To establish the estimates \eqref{fullp}--\eqref{as2-new},
we represent the resolvent in terms of Jost functions.
In particular,  we use this representation to prove the limiting absorption principle 
and to establish absence of embedded eigenvalues in the continuous spectrum
under the condition $q^{ij}\in\ell^1$.

The dispersive decay for $L^1\to L^\infty$ with decay rate $t^{-1/2}$ for continuous perturbed Schr\"odinger equations
has been established by Weder \cite{wed} with later improvements by Goldberg and Schlag \cite{GS}, and by
Egorova, Marchenko, and us in \cite{EKMT}.

The dispersive decay of type \eqref{fullp} for discrete Schr\"odinger equation
has been established by Pelinovsky and Stefanov \cite{PS}
under the assumption that there are no resonances and under more restrictive conditions
on the potential. Further developments were given by Cuccagna and Tarulli \cite{CT}.
In our recent paper \cite{EKT} we weakened the conditions of \cite{CT} in the resonant case.
Moreover, in \cite{EKT} and \cite{EKMT} we obtain the analogous decay for discrete
and continuous Klein--Gordon equations.
For the discrete Dirac equation \eqref{Dir} the decay \eqref{fullp} has not been obtain previously.

The decay of type \eqref{as2-new} in weighted norms for the one-dimensional continuous Schr\"odinger  equation in the non-resonant case
has been established by Jensen and Nenciu \cite{JN} and, for more general PDEs of Schr\"odinger type, by Murata \cite{M}.
For the one-dimensional Klein--Gordon equation the analogous decay in weighted energy norms
has been obtained by Komech and one of us \cite{KK} (see also the survey \cite{K10}). For discrete Schr\"odin\-ger and Klein--Gordon equations with compactly
supported potentials it has been obtained in \cite{kkk} and generalized in \cite{PS} to discrete Schr\"odinger
equation with non-compactly supported potentials $q$ under the decay condition 
$|q_n|\le (1+|n|)^{-\beta}$ with $\beta>5$ for $\sigma>5/2$.
In \cite{EKT},  we improved this result to $q\in\ell^1_2$ and $\sigma>3/2$.

For the continuous one-dimensional Dirac equation the dispersive decay of type \eqref{as2-new} with
$\sigma>5/2$  has been obtained in \cite{K11}. 
For the discrete Dirac equation the decay \eqref{as2-new} is again new.

Asymptotics of the type \eqref{fullp}--\eqref{as2-new} can be applied in proving asymptotic
stability of solitons for the associated discrete one-dimensional nonlinear Dirac equations. 

\section{The free discrete Dirac equation}
\label{free-sect}

First we consider the free equation \eqref{Dir} with $Q=0$.
We have
\begin{equation}\label{Rid}
(\D_0-\lambda)(\D_0+\lambda)=
-\Delta_L+m^2-\lambda^2,
\end{equation}
where $\Delta_L$ is the discrete Laplacian given by
\begin{equation*}
  (\Delta_L u)_n=u_{n+1}-2u_n+u_{n-1}, \quad n\in\Z.
\end{equation*}
Denote by ${\cR}_0(\lambda)=({\D}_0-\lambda)^{-1}$
the resolvent of the  free Dirac operator $\D_0$.
Then \eqref{Rid} implies
\begin{equation}\label{RD}
{\cR}_0(\lambda)=(\D_0+\lambda)R_{0}(\lambda^2-m^2),
\end{equation}
where $R_{0}(\omega)=(-\Delta_L-\omega)^{-1}$ 
is the resolvent of  operator $-\Delta_L$.
Adopting the notation $[K]_{n,k}$ for the kernel of an operator $K$, that is,
\[
(K u)_n=\sum_{k\in\Z} [K]_{n,k}u_k,\quad n\in\Z,
\]
the kernel of the resolvent $R_0(\omega)$ is given by  (see \cite{kkk,tjac})
\begin{equation} \label{R02}
[R_0(\omega)]_{n,k} =\frac 1{2\pi}\int\limits_{\mathbb{T}}
\frac{\E^{-\I\theta(n-k)}}{\phi(\theta)  -\omega} d\theta
= \frac{\E^{-\I\theta(\omega)|n-k|}}{2\I\sin\theta(\omega)},
\quad\omega\in\C\setminus [0,4],
\end{equation}
$n,k\in\Z$. Here $\theta(\omega)$ is the unique solution of the equation
\begin{equation}\label{theta}
  2-2\cos\theta=\omega,\quad\theta\in \Sigma:=\{ -\pi\le\re\theta\le\pi,\;\im\theta< 0\}/2\pi\Z.
\end{equation}
Observe that $\theta\mapsto\omega=2-2\cos\omega$ is a biholomorphic map from $\Sigma\to\C\setminus [0,4]$.

 The kernel of the free Dirac propagator  can be easily
computed using the spectral theorem and formulas \eqref{RD}--\eqref{R02}
\begin{align*}
[\E^{-\I t\D_0}]_{n,k}&=\frac 1{2\pi\I}\int\limits_{\Gamma}
\E^{-\I t\lambda}[\cR_0(\lambda+\I0)-\cR_0(\lambda-\I0)]_{n,k}\,d\lambda\\ 
&=-\frac 1{4\pi}\int\limits_{\Gamma} \E^{-\I t\lambda}
\begin{pmatrix}
m+\lambda   & d\\
d^* & -m+\lambda
\end{pmatrix}
\Big(\frac{\E^{-\I\theta_+|n-k|}}{\sin\theta_+}-
\frac{\E^{-\I\theta_-|n-k|}}{\sin\theta_-}\Big)d\lambda,
\end{align*}
where 
\begin{equation}\label{theta_pm}
\theta_+=\theta((\lambda+\I 0)^2-m^2)\in [-\pi,0],
\quad\theta_-=\theta((\lambda-\I 0)^2-m^2)\in [0,\pi],\quad\lambda\in\overline\Gamma.
\end{equation}
Finally, 
\begin{equation}\label{Om}
[\E^{-\I t\D_0}]_{n,k}=\frac 1{2\pi}\sum\limits_{j=-1}^1 \int_{-\pi}^{\pi}
\frac{\Omega_j(\theta)}{g(\theta)}\,
\E^{-\I tg(\theta)}\, \E^{-\I\theta|n-k+j|}d\theta,
\end{equation}
where $g(\theta):=\sqrt{2-2\cos\theta+m^2}$ and
\[
\Omega_{-1}(\theta)=\left(\!\!\begin{array}{cc}
0  & 0\\
-1 & 0
\end{array} \!\!\right)\!,\quad\Omega_{1}(\theta)=\left(\!\!\begin{array}{cc}
0  & -1\\
0 & 0
\end{array} \!\!\right)\!,\quad\Omega_0(\theta)=\left(\!\!\begin{array}{cc}
m+g(\theta)  & 1\\
1 & -m+g(\theta)
\end{array} \!\!\right).
\]
For the free discrete Dirac equation the $\bl^1\to \bl^{\infty}$ decay of type \eqref{fullp}
holds. However,  the $\bl^2_{\sigma}\to \bl^2_{-\sigma}$ decay
holds with rate  $t^{-1/2}$ only (as in the  continuous case).
This is caused by the presence of resonances at the edge points $\mu=\pm m,\pm\sqrt{m^2+4}$.

\begin{proposition}\label{free-as}
The following asymptotics hold:
\begin{equation}\label{D0-as1}
\Vert \E^{-\I t\D_0}\Vert_{\bl^1\to \bl^{\infty}}=\mathcal{O}(t^{-1/3}),\quad t\to\infty,
\end{equation}
\begin{equation}\label{D0-as2}
\Vert \E^{-\I t\D_0}\Vert_{\bl^2_\sigma\to \bl^2_{-\sigma}}=\mathcal{O}(t^{-1/2}),
\quad t\to\infty,\quad\sigma>1/2.
\end{equation}

\end{proposition}
\begin{proof}
It suffices to consider the operator $K(t)$ with the kernel
\begin{equation}\label{cK}
[K(t)]_{n,k}=\int_{-\pi}^{\pi} \psi(\theta)
\E^{-\I tg(\theta)}\, \E^{-\I\theta|n-k|}d\theta,
\end{equation}
where $\psi(\theta)$ is some smooth function, and obtain the asymptotics
\begin{equation}\label{K-as1}
\Vert K(t)\Vert_{\ell^1\to \ell^{\infty}}=\mathcal{O}(t^{-1/3}),\quad t\to\infty,
\end{equation}
\begin{equation}\label{K-as2}
\Vert K(t)\Vert_{\ell^2_\sigma\to \ell^2_{-\sigma}}=\mathcal{O}(t^{-1/2}),
\quad t\to\infty,\quad\sigma>1/2.
\end{equation}
\smallskip\\
{\it Step i)}
Since
\begin{equation}\label{vert}
\Vert K(t)\Vert_{\ell^1\to\ell^\infty}=\sup\limits_{\Vert f\Vert_{\ell^1}=1,\Vert g\Vert_{\ell^1}=1}
\langle f, K(t)g\rangle
\le \sup\limits_{n,k\in\Z}|[K(t)]_{n,k}|,
\end{equation}
then for \eqref{K-as1} it suffices to prove that
\begin{equation}\label{K-as11}
\sup_{n,k\in Z}|[K(t)]_{n,k}|\le C t^{-1/3}.
\end{equation}
Abbreviate $v:=\frac{|n-k|}t\ge 0$ and set $\varkappa=(2+m^2-\sqrt{4m^2+m^4})/2$,
$0<\varkappa<1$.
It is easy to check that for $v\not= v_0:=\sqrt\varkappa$ the phase function
\begin{equation}\label{phi}
\Phi_v(\theta)=g(\theta)+v\theta
\end{equation}
 has at most two non-degenerate stationary points.
In the case  $v=v_0$ there exists a unique degenerate stationary point $\theta_0=-\arccos\varkappa$,
$-\pi/2<\theta_0<0$,
such that  $\Phi'''(\theta_0)=\sqrt{\varkappa}\not =0$. 
Hence \eqref{K-as11} follows from the following lemma:
\begin{lemma}\label{lem:vC0} (cf.\ \cite{St})
Consider the oscillatory integral
\[
I(t) = \int_a^b \psi(\theta) \E^{\I t \phi(\theta)} d\theta, \qquad -\pi \le a < b \le \pi,
\]
where $\phi(\theta)$ is real-valued smooth function, and $|\psi(\theta)|+|\psi'(\theta)|\le M$.
If $|\phi^{(k)}(\theta)| >0$, $\theta\in [a,b]$, for some $k\ge 2$  then
\[
|I(t)| \le C_k (M)\big(t\min_{[a,b]}|\phi^{(k)}(\theta)|\big)^{-1/k}, \quad t\ge 1.
\]
\end{lemma}
{\it Step ii)}
The norm of the operator $K(t):\ell^2_\sigma\to\ell^2_{-\sigma}$ is equivalent to the 
norm of the operator $K_\sigma(t)=(1+|n|)^{-\sigma}K(t)(1+|k|)^{-\sigma}:\ell^2\to\ell^2$.
Hence for \eqref{K-as2} it suffices to prove that the Hilbert-Schmidt norm of $K_\sigma(t)$
does not exceed $Ct^{-1/2}$, i.e.
\begin{equation}\label{GS}
\Big[\sum_{n,k\in Z}\frac {([K(t)]_{n,k})^2}{(1+|n|)^{2\sigma}(1+|k|)^{2\sigma}}\Big]^{1/2}
\le Ct^{-1/2},\quad\sigma>1/2.
\end{equation}
We divide the domain of integration in \eqref{cK} into the domains
\begin{equation}\label{Jpm}
{\bf J}_{\pm}=\{\theta:\:|\theta\mp\theta_0|\le \nu|\theta_0|\},\quad
{\bf J}=[-\pi, \pi]\setminus ({\bf J}_+\cup {\bf J}_-),
\end{equation}
where $\nu=\nu(m)\in(0,1/2]$, will be specified below.
Since  $|\Phi''(\theta)|\ge C(\nu)$ for $\theta\in{\bf  J}$, we infer
\[
\sup_{n,k\in\Z}\big|\int_{\bf J} \psi(\theta) \E^{-\I t\Phi_v(\theta)}d\theta\big|\le Ct^{-1/2},
\quad t\ge 1,
\]
from the van der Corput Lemma~\ref{lem:vC0}. Consequently, \eqref{GS} for the part over $\bf J$  follows. 
For the part over ${\bf J}_-$ we apply integration by parts to obtain
\[
\sup_{n,k\in\Z}\big|\int_{{\bf J}_-} \psi(\theta) \E^{-\I t\Phi_v(\theta)}d\theta\big|\le Ct^{-1},
\quad t\ge 1,
\] 
from which \eqref{GS} for the part over ${\bf J}_-$ follows. 

It remains to consider the part over ${\bf J}_+$.
For any fixed  $\sigma>1/2$,  there exist an integer $N>0$ such that
\begin{equation}\label{sigma}
\sigma>1/2+(1/2)^{N}.
\end{equation}
Denote $t_j=t^{-(\frac 12)^j}$, $1\le j\le N$,  $t_0=0$, $t_{N+1}=\nu|\theta_0|$.
We further divide the domain ${\bf J}_+$ into subdomains  
${\bf J}_+^j=t_j\le |\theta-\theta_0|\le t_{j+1}$, 
$0\le j\le N$. For the integral over ${\bf J}_+^0$ the estimate \eqref{GS} evidently holds.
It remains to get \eqref{GS} for the operators $K_j(t)$ with kernels
\begin{equation}\label{Kj-def} 
[K_j(t)]_{n,k}=\int_{{\bf J}_+^j} \psi(\theta) \E^{-\I t\Phi_v(\theta)}d\theta,\quad 1\le j\le N.
\end{equation}
By the van der Corput lemma
\begin{equation}\label{C7}
\sup_{n,k\in\Z}|[K_j(t)]_{n,k}|
\le Ct^{-1/2}\left(\min_{\theta\in J_+^j}|\Phi''_v(\theta)|\right)^{-1/2}
\!\!\!\le C(tt_j)^{-1/2}.
\end{equation}
Now we choose $\ve=2^{-N}$, so that $t^{\ve}=t^{-1}_N$, and  consider $|v_0-v|\le \frac 12 v_0 t_j t^\ve$ and $|v_0-v|\ge \frac 12 v_0 t_j t^\ve$ separately.
\smallskip

In the first case, we take $T_j=\left\{(n,k)\in\Z^2:\ |v_0t-|n-k||\le \frac 12v_0 t_j t^{1+\ve}\right\}$
as the domain of summation. Since this domain is symmetric  with respect to the
map $(n,k)\mapsto (-n, -k)$, we can make the change of variables $p=n-k$, $q=n+k$ and estimate
\[
b_j(t):=\sum_{(n,k)\in T_j}\frac {1}{(1+|n|)^{2\sigma}(1+|k|)^{2\sigma}}
\]
by
\[
b_j(t)\le \sum_{q\in\Z}\sum_{p=\ceil{v_0(t-\frac 12t_j t^{1+\ve})}}^{\floor{v_0(t+\frac 12t_j t^{1+\ve})}}
\frac {2}{(1+\frac{1}{2}|p+q|)^{2\sigma}(1+\frac{1}{2}|p-q|)^{2\sigma}},
\]
where $\floor{\cdot}$ and $\ceil{\cdot}$ denote the usual floor and ceiling functions, respectively.
The sum with respect to $p$  is finite with the number of summands less then
$\floor{v_0t_j t^{1+\ve}}+2$. Since $t_j t^{1+\ve}\le t$ for $j=1,\dots, N$,
 we have $p\geq \frac 12 v_0t$ in the domain of summation.
Consequently $p+q\geq \frac 12 v_0t$ for $q\geq 0$ and $p-q\geq \frac 12 v_0t$ for $q<0$.
Using these estimates and interchanging the order of summation, we get
\begin{equation}\label{GGG}
 b_j(t)\leq C\,\frac{\floor{t_j t^{1+\ve}}}{t^{2\sigma}}\leq C t_j t^{1+\ve-2\sigma}.
\end{equation}
Thus, by \eqref{C7}, \eqref{GGG}, and \eqref{sigma}
\begin{equation}\label{GG}
\sum_{n,k\in T_j}\frac {([K_j(t)]_{n,k})^2}{(1+|n|)^{2\sigma}(1+|k|)^{2\sigma}}
\leq\sup_{n,k\in \Z}([K_j(t)]_{n,k})^2 b_j(t)\leq C t^{-2\sigma + \ve}\leq C t^{-1}
\end{equation}
implying
\begin{equation}\label{C8}
\Vert K_j(t)\Vert_{\ell^2_\sigma\to\ell^2_{-\sigma}}\le Ct^{-1/2},\quad j=1,\dots,N,
\quad\sigma>1/2,
\end{equation}
in the first case.

\smallskip

In the second case, we apply integration by parts. To this end we have to estimate:
(a) $|\Phi^\prime_v(\theta)|^{-1}$ at the points $\theta_0\pm t_j$ and $\theta_0 \pm t_{j+1}$,  and
(b) the integral of the function $|\Phi^{\prime\prime}_v(\theta)|(\Phi^\prime_v(\theta))^{-2}$ between these points.
But since the function $\Phi^{\prime\prime}_v(\theta)$ does not change its sign on the intervals
$[\theta_0+t_j, \theta_0+t_{j+1}]$ and $[\theta_0-t_{j+1},\theta_0-t_j]$, the antiderivative of
$|\Phi^{\prime\prime}_v(\theta)|(\Phi^\prime_v(\theta))^{-2}$ is equal up to a sign to $(\Phi^\prime_v(\theta))^{-1}$.
Thus it is sufficient to consider the case (a) only.

We have ${\Phi}_v(\theta)=g(\theta)+v\theta$, therefore
\[
\Phi^\prime_v(\theta)=g^\prime(\theta)+v
=g^\prime(\theta_0)+\frac 12 g^{\prime\prime\prime}(\tilde\theta)(\theta - \theta_0)^2+v
=\frac 12 g^{\prime\prime\prime}(\tilde\theta)(\theta - \theta_0)^2+v-v_0.
\]
Here we used $g^{\prime}(\theta_0)=-v_0$ and $g^{\prime\prime}(\theta_0)=0$.
Hence for large $t$
\[
|\Phi^\prime_v(\theta_0\pm t_{j+s})|\ge |v-v_0|-Ct^2_{j+s}\ge t_j(\frac {v_0t^\ve}2-C)\ge C_1t_j, 
\quad j=1,\dots, N-1,\quad s=0,1
\]
and then
\begin{equation}\label{denom}
|[K_j(t)]_{n,k}|\le Ct^{-1}t_j^{-1}\le C t^{-1/2}, \quad j=1,\dots, N-1.
\end{equation}

In the case $j=N$, we have $|v-v_0|\ge \frac 12 v_0$. Furthermore,
\[
\Phi^\prime_v(\theta_0\pm t_{N+1})
=\frac 12 g^{\prime\prime\prime}(\tilde\theta)(\nu|\theta_0|)^2+v-v_0.
\]
Since $|g^{\prime\prime\prime}(\theta)|\le G=G(\mu)$, $\theta\in [-\pi,\pi]$,
then we can choose 
$\nu=\min\{\frac 12, \sqrt{\frac{2v_0}{3G\theta_0^2}}\,\}$ to obtain
$|\Phi^\prime_v(\theta_0\pm t_{N+1})|\ge \frac 16 v_0$.
Respectively, $|\Phi^\prime_v(\theta_0\pm t_{N+1})|^{-1}\le 6/v_0$, and hence
\begin{equation}\label{denom1}
|[K_N(t)]_{n,k}|\le Ct^{-1}.
\end{equation}
Taking into account \eqref{denom} and \eqref{denom1}, we obtain \eqref{C8} also in this case.
\end{proof}

\section{Jost solutions}
\label{Jost-func}

In this section, we establish basic properties of the Jost functions. For related results in the special case of a diagonal potential, we refer to \cite{BAO}, \cite{AO}.
In the special case of supersymmetric  operators these results can also be inferred from the corresponding results for Jacobi operators (cf.\ \cite[Sect.\ 14.3]{tjac}).

Denote by $\Gamma=(-\sqrt{m^2+4},-m)\cup (m,\sqrt{m^2+4})$, $\Xi:=\C\setminus\overline\Gamma$, and
let $\Xi_+=\{\lambda\in\Xi,\,\re\lambda\ge 0\}$.
For any $\lambda\in \overline\Xi_+$ consider the Jost solutions ${\bf w}=(u, v)$ to 
\begin{equation} \label{D1}
{\D}{\bf w}=\lambda{\bf w}
\end{equation}
satisfying the boundary conditions
\begin{equation}\label{J2}
{\bf w}_n^{\pm}(\theta)=
\begin{pmatrix}
u_n^{\pm}(\theta)\\
v_n^{\pm}(\theta)
\end{pmatrix}
\to
\begin{pmatrix}
1\\
\alpha_{\pm}(\theta)
\end{pmatrix}e^{\mp\I\theta n},\quad n\to\pm\infty,
\end{equation}
where 
\[
\alpha_{\pm}(\theta)=\frac{1-e^{\pm\I\theta}}{m+\lambda},
\]
and $\theta=\theta(\lambda)\in\overline\Sigma$ solves
\[
2-2\cos\theta=\lambda^2-m^2.
\]
The boundary condition \eqref{J2} arise naturally in \eqref{D1}  with
$Q \equiv 0$.
For nonzero $Q$ with $q^{ij}\in \ell_1^1$
the Jost solution exists everywhere in $\overline\Xi_+$, but for 
$q^{ij}\in\ell^1$ it exists only away from the edges of continuous spectrum.
Introduce
\begin{equation}\label{Jostcut}
{\bf h}^\pm_n(\theta)=\E^{\pm \I n \theta}{\bf w}^\pm_n(\theta)
\end{equation}
and set
\begin{align*}
\overline\Sigma_M:&=\{\theta\in\overline\Sigma:|\im\theta|\le M\}, \quad M\ge 1,\\
\overline\Sigma_{M,\delta}:&=\{\theta\in\overline\Sigma_{M}:\, |\E^{-\I\theta} \pm 1|>\delta
\},\quad 0<\delta<\sqrt 2.
\end{align*}
\begin{proposition}\label{Jost-sol}
(i) Let $q^{ij}\in\ell^1_s$ with $s=0,1,2$.
Then the functions ${\bf h}^\pm_n(\theta)$ can be differentiated $s$ times on $\overline\Sigma_{M,\delta}$, and the following estimates hold:
\begin{equation}\label{dh-est}
|\frac{\partial^p}{\partial \theta^p}{\bf h}^\pm_n(\theta) |\le C(M,\delta)\max ((\mp n)|n|^{p-1}, 1),  ~~~n\in\Z,~~~ 0\le p\le s,~~~ \theta\in \overline\Sigma_{M,\delta}.
\end{equation}\\
(ii) If additionally  $q^{ij}\in \ell_{s+1}^1$, then ${\bf h}^\pm_n(\theta)$ can be differentiated $s$ times on $\overline\Sigma_M$, and the following estimates hold:
\begin{equation}\label{Jostderiv}
|\frac{\partial^{p}}{\partial \theta^{p}}{\bf h}^\pm_n(\theta) |\le C(M)\max ((\mp n)|n|^{p}, 1),\quad n\in\Z,\quad 0\le p\le s,\quad\theta\in\overline\Sigma_M.
\end{equation}
\end{proposition}
\begin{proof}
The Green's function representation for the solutions of \eqref{D1} reads:
\begin{equation}\label{JR}
{\bf h}_n^{\pm}(\theta)=\begin{pmatrix}
1\\
\alpha_{\pm}(\theta)
\end{pmatrix}+\sum\limits_{k=n}^{\pm\infty}
G^{\pm}(k-n,\theta)Q_k{\bf h}_k^{\pm}(\theta),
\end{equation}
where
\begin{eqnarray}\nonumber
G^{\pm}(l,\theta)\!\!\!&=&\!\!\!\frac{(m+\lambda)}{2\I\sin\theta}\begin{pmatrix}
1-\E^{\mp 2\I\theta l} & \alpha_{\mp}-\alpha_{\pm}\E^{\mp 2\I\theta l}\\
\alpha_{\pm}-\alpha_{\mp}\E^{\mp 2\I\theta l}       
&(1-\E^{\mp 2\I\theta l})\frac{\lambda-m}{m+\lambda}
\end{pmatrix},\quad \pm l\ge 1,\\
\nonumber\\
\nonumber
G^{+}(0,\theta)\!\!\!&=&\!\!\!\begin{pmatrix}
0  &  0 \\
-1 &  0
\end{pmatrix},\quad
G^{-}(0,\theta)=\begin{pmatrix}
0     & -1 \\
0     & 0 
\end{pmatrix}.
\end{eqnarray}
We consider the case $``+"$ only since in the $``-"$ case the proof is similar. 
Abbreviate ${\bf h}_n(z)={\bf h}_n^+(\theta)$ with $z=\E^{-\I\theta}$.
Equation \eqref{JR} implies
\begin{equation}\label{green}
A_n{\bf h}_n(z)=\begin{pmatrix}
1\\\alpha_+
\end{pmatrix}
+\sum\limits_{k=n+1}^{\infty} G(k-n,z)Q_k{\bf h}_k(z),
\end{equation}
where
\[
G(l,z)=\frac{(m+\lambda)}{z^{-1} - z}\begin{pmatrix}
1-z^{2l}& \alpha_{+}-\alpha_{-}z^{2l}\\
\alpha_{-}-\alpha_{+}z^{2l}  &(1-z^{2l})\frac{\lambda-m}{m+\lambda}
\end{pmatrix},\quad A_n=\begin{pmatrix}
1 & 0 \\ q_n^{11} &1+q_n^{12}\end{pmatrix},
\]
and $\alpha_{\pm}=\alpha_{\pm}(z)=(1-z^{\mp 1})/(m+\lambda)$.

For $\theta\in \overline\Sigma_{M,\delta}$, we have  $|z^2 -1|\geq C(\delta)>0$. Then
\[
|G(l,z)|\leq \frac{C(M)}{|z^2 - 1|}\leq C(M,\delta),\quad l> 0,
\]
and the method of successive approximations (cf.\ \cite{tjac}) implies $|{\bf h}_n(z)|\leq C(M,\delta)$.
Then \eqref{dh-est} with $p=0$ follows.
Furthermore,
\begin{equation}
\label{ddots}|\frac{d^p}{dz^p}  G(l,z)|\leq C(M,\delta) (k-n)^p ,\quad p\ge 1,\quad l> 0,
\quad \theta\in \overline\Sigma_{M,\delta}.
\end{equation}
\smallskip\\
Now let $q^{ij}\in\ell^1_1$. Consider the first derivative of ${\bf h}_n(z)$. We have
\begin{equation}\label{appr}
A_n\frac{d}{dz} {\bf h}_n(z)= \begin{pmatrix}
0\\\frac{d}{dz}\alpha_{+}
\end{pmatrix}+\phi_n(z)+\sum_{m=n+1}^\infty  G(k-n,z)Q_k \frac{d}{dz}{\bf h}_k(z),
\end{equation}
where
\[
\phi_n(z):=\sum_{k=n+1}^\infty\frac{d}{dz}  G(k-n,z)Q_k {\bf h}_k(z).
\]
Moreover, we we have the estimate $|\phi_n(z)|\leq C(M,\delta)$ for $ n\geq 0$ and $\theta\in \overline\Sigma_{M,\delta}$
by \eqref{dh-est} with $p=0$ and \eqref{ddots}.
Applying the method of successive approximations  to \eqref{appr}, we get \eqref{dh-est} with $p=1$.
For the case $p= 2$, we proceed in the same way.

The estimate \eqref{Jostderiv} can be obtained from \eqref{green} by the same approach by virtue of the estimate
$|\frac{d^p}{dz^p} G(l,z)|\leq C(M)l^{p+1}$, which is valid for  $l>0$.
\end{proof}
\begin{corollary} In the case $q^{ij}\in l^1$ Proposition \ref{Jost-sol} (i)  implies in particular that for any
$\theta\in\overline\Sigma\setminus\{0;\pm\pi\}$
we have the estimate $|{\bf h}_n^\pm(\theta)|\leq C(\theta)$ for all $n\in\Z$.
Here $C(\theta)$ can be chosen uniformly on compact subsets of $\overline{\Sigma}$ avoiding the band edges.
Together with \eqref{Jostcut} this implies
\begin{equation}\label{De0}
| {\bf w}^{\pm}_n(\theta) |\le C(\theta) \E^{\pm \im(\theta)n}\, ,
\quad \theta\in\overline\Sigma\setminus\{0;\pm\pi\},\quad n\in\Z.
\end{equation}
\end{corollary}

Now we define the Jost function for $\re\lambda\le 0$.
Similarly to the analysis for $\re\lambda\ge 0$, we consider solutions of system \eqref{D1}
according the boundary conditions
\begin{equation}\label{tJ2}
\tilde {\bf w}_n^{\pm}(\theta)=
\begin{pmatrix}
\tilde u_n^{\pm}(\theta)\\
\tilde v_n^{\pm}(\theta)
\end{pmatrix}
\to
\begin{pmatrix}
\tilde\alpha_{\pm}(\theta)\\
1
\end{pmatrix} \E^{\mp\I\theta n},\quad n\to\pm\infty,
\end{equation}
where 
\[
\tilde\alpha_{\pm}(\theta)=\frac{1-\E^{\mp\I\theta}}{\lambda-m}.
\]
Using a similar Green's function, Propositions \ref{Jost-sol}
can be extended to functions $\tilde {\bf w}_n^{\pm}(\theta)$. In particular, 
for  $\tilde {\bf w}_n^{\pm}(\theta)$ and
$\tilde{\bf h}_n^{\pm}(\theta)=\tilde{\bf w}_n^{\pm}(\theta)\E^{\pm \I\theta n}$
the bounds \eqref{dh-est}, \eqref{Jostderiv} and \eqref{De0} hold.

\section{Wronskians}
\label{Wr-sect}

As before we consider the case $\re\lambda\ge 0$ only.
Denote by  $W({\bf w}^1,{\bf w}^2)$  the Wronski determinant of any two solutions
${\bf w}^1$ and ${\bf w}^2$ to  \eqref{D1}:
\[
W({\bf w}^1,{\bf w}^2):=\left|\begin{array}{cc}
 u^1_n & u^2_n\\
 v^1_{n+1}    &  v^2_{n+1}
\end{array}\right|.
\]
It is easy to check that if $q^{12}\equiv q^{21}$ then $W({\bf w}^1,{\bf w}^2)$ is constant in $n\in\Z$
for arbitrary solutions ${\bf w}^1$ and ${\bf w}^2$ of \eqref{D1}.
The Jost solutions ${\bf w}^{\pm}(\theta)$ and ${\bf w}^{\pm}(-\theta)$ 
are independent for $\theta\in (-\pi,0)\cup(0,\pi)$ since 
\begin{equation}\label{W1}
W({\bf w}^{\pm}(\theta),{\bf w}^{\pm}(-\theta))=\left|\begin{array}{cc}
 1 & 1\\
\alpha_{\pm}(\theta)e^{\mp i\theta}     & \alpha_{\pm}(-\theta)e^{\pm\I\theta} 
\end{array}\right|=\pm\frac{2\I\sin\theta}{m+\lambda}\not =0,
\end{equation}
if $\theta\not =0,\pm\pi$.
Then there exist (unique) functions $a_{\pm}(\theta)$ and $b_{\pm}(\theta)$
such that
\begin{equation}\label{scatr}
{\bf w}^{\pm}(\theta)=a_{\mp}(\theta){\bf w}^{\mp}(-\theta)+b_{\mp}(\theta){\bf w}^{\mp}(\theta).
\end{equation}
Calculating the Wronskians, for $\theta\in (-\pi,0)\cup(0,\pi)$ we obtain 
\begin{equation}\label{a+}
a_-(\theta)=a_+(\theta)=\frac{W(\theta)}{W({\bf w}^-(-\theta),{\bf w}^-(\theta))}=
\frac{W(\theta)(m+\lambda)}{2\I\sin\theta}
\end{equation}
and
\begin{equation}\label{b+}
b_{\pm}(\theta)=\pm\frac{W^{\pm}(\theta)}
{W({\bf w}^-(\theta),{\bf w}^-(-\theta))}=\pm\frac{W^{\pm}(\theta)(m+\lambda)}{2\I\sin\theta},
\end{equation}
where
\[
W(\theta)=W({\bf w}^+(\theta),{\bf w}^-(\theta)),\quad
W^{\pm}(\theta)=W({\bf w}^{\mp}(\theta),{\bf w}^{\pm}(-\theta)).
\]
\begin{lemma}\label{W} Let $q^{ij}\in\ell^1$. Then
$W(\theta)\not =0$ for $\theta\in(-\pi,0)\cup(0,\pi)$.
\end{lemma}
\begin{proof}
Since ${\bf w}^+(-\theta)=\overline{\bf w}^+(\theta)$ then (\ref {scatr}) implies 
\[
{\bf w}^+(-\theta)=\overline b_-(\theta){\bf w}^-(-\theta)+\overline a_-(\theta){\bf w}^-(\theta).
\]
Therefore, 
\[
W({\bf w}^+(\theta),{\bf w}^+(-\theta))=(|a_-(\theta)|^2-|b_-(\theta)|^2)
W({\bf w}^-(-\theta),{\bf w}^-(\theta)).
\]
Hence, from \eqref{W1} it follows that
\begin{equation}\label{ab}
|a_-(\theta)|^2-|b_-(\theta)|^2=1.
\end{equation}
Now the assertion of the lemma follows from \eqref{a+} and \eqref{b+}.
\end{proof}

\begin{remark}
Proposition \ref{Jost-sol} and Lemma \ref{W} eliminates the possibility of embedded eigenvalues 
in the continuous spectrum
of $\D$ for $q^{ij}\in\ell^1$ because the space of solutions of the Dirac system \eqref{D1} for 
$\lambda\in (-\sqrt{4+m^2},-m)$ and $\lambda\in (m,\sqrt{4+m^2})$
is spanned by the two fundamental solutions $\tilde{\bf w}_n^{\pm}(\theta)$ and ${\bf w}_n^{\pm}(\theta)$
which are not square summable near $n\to\pm\infty$.
\end{remark}
Now we discuss an alternative definition of resonances.
\begin{definition}
For $\lambda\in\{m,\sqrt{4+m^2}\}$
any nonzero solution ${\bf w}\in\bl^{\infty}$  of the equation  $\D{\bf w}=\lambda {\bf w}$
is called a resonance function,
and in this case  $\lambda$  is called a resonance.
\end{definition}

\begin{lemma}\label{W0}
Let $q^{ij}\in\ell^1_1$. Then $\lambda=m$ (or~$\lambda=\sqrt{4+m^2}$) is a resonance
if and only if $W(0)=0$ (or $W(\pi)=0$).
\end{lemma}
\begin{proof}
{\it Step i)}
In the case $\lambda=m$, we have
\begin{equation}\label{w0}
{\bf w}^\pm_n
= \begin{pmatrix}
u_n^{\pm}\\ v_n^{\pm}
\end{pmatrix}
= \begin{pmatrix}
1\\0
\end{pmatrix} +o(1),\quad n\to\pm\infty.
\end{equation}
Introduce another solution ${\bf w}^*=(u^*,v^*)$ to  \eqref{D1} satisfying 
\begin{equation}\label{w*}
W({\bf w}^+,{\bf w}^*)=u_n^+v_{n+1}^*-u_n^*v_{n+1}^+=1. 
\end{equation}
If the sequence $u_n^*$ is bounded for  positive $n$, then from
\eqref{w0}--\eqref{w*} it follows that $v_{n}^* \to 1$ for $n\to +\infty$.
Then the second line of equation \eqref{D1}  with $\lambda=m$  implies
\[
u^*_{n+1}-u^*_{n}=mv^*_{n+1}-q^{21}_{n+1}u^*_{n+1}-q^{22}_{n+1}v^*_{n+1}\to m,\quad n\to +\infty.
\]
Hence, for  sufficiently large positive $n_0$, we obtain
\[ 
u^*_{n_0+k}=u^*_{n_0}+km+o(k),\quad k\to+\infty,
\]
which contradicts our assumption on the boundedness of $u_n^*$.
Since
${\bf w}^- = \alpha {\bf w}^+ + \beta {\bf w}^*$, then it is a bounded solution if and only if 
$\beta = W({\bf w}^+,{\bf w}^-) =0$.
\smallskip\\
{\it Step ii)} Consider now the case  $\lambda=\sqrt{m^2+4}$. In this case
\begin{equation}\label{w-pi}
{\bf w}^\pm_n
= \begin{pmatrix}
u_n^{\pm}\\v_n^{\pm}
\end{pmatrix}
= (-1)^n \begin{pmatrix}
1\\2/\tilde m
\end{pmatrix} +o(1),\quad n\to\pm\infty.
\end{equation}
where $\tilde m=m+\sqrt{m^2+4}$. Introduce another solution ${\bf w}^*=(u^*,v^*)$ to  \eqref{D1} satisfying
\eqref{w*} and suppose that the sequence $u_n^*$ is bounded for  positive $n$, 
then the sequence $v_n^*$ is also bounded due to \eqref{w-pi}.
Then \eqref{w*} implies that
\[
\tilde mv^*_{n+1}=-2u^*_{n}+(-1)^n\tilde m+o(1),\quad n\to +\infty.
\]
Now the second line of equation \eqref{D1}  with $\lambda=\sqrt{m^2+4}$ yields 
\[
u^*_{n+1}=-u^*_{n}+(-1)^n\tilde m+o(1),\quad n\to +\infty.
\]
Hence, for  sufficiently large positive $n_0$, we obtain
\[ 
u^*_{n_0+k}=(-1)^ku^*_{n_0}+k(-1)^{n_0+k}\tilde m+o(k),\quad k\to+\infty,
\]
which contradicts our assumption on the boundedness of $u_n^*$. Since
${\bf w}^- = \alpha {\bf w}^+ + \beta {\bf w}^*$, then it is a bounded solution if and only if 
$\beta = W({\bf w}^+,{\bf w}^-) =0$.
\end{proof}
\begin{remark}\label{W/W}
From \eqref{ab} it follows that the zeros of the Wronskians $W(\theta)$ and $W^{\pm}(\theta)$
at the points $0,\pi$ can be at most of first order. 
\end{remark}
For $\re\lambda\le 0$ the corresponding Wronskian $\tilde W(\theta)$ has the same  properties.

\section{The limiting absorption principle}
\label{LAP-D}

Given the Jost solutions, we can express the kernel of the resolvent 
$\cR(\lambda):\bl^2 \to \bl^2$ for $\lambda\in \C\setminus\overline\Gamma$.
Recall that $\theta\mapsto\omega(\theta)$ is a biholomorphic map $\Sigma\to\Xi$.
The method of variation of parameters gives
\begin{lemma}\label{cR-rep}
Let $q^{ij}\in\ell^1$. Then
for any $\lambda\in \C\setminus\overline\Gamma$, $\re\lambda\ge 0$, the operators
$\cR(\lambda):l^2\to l^2$ 
can be represented by the integral kernel as follows
\begin{equation}\label{RJ-rep}
[\cR(\lambda)]_{k,n}=\frac{1}{W(\theta(\lambda))}
\begin{cases}
{\bf w}_k^-(\theta(\lambda))\otimes{\bf w}_n^+(\theta(\lambda)),\quad k\le n,\\
{\bf w}_k^+(\theta(\lambda))\otimes{\bf w}_n^-(\theta(\lambda)),\quad k\ge n,
\end{cases}
\end{equation}
where
\[
{\bf w}_k^1\otimes{\bf w}_n^2=\begin{pmatrix}
u_k^1u_n^2& v_{k+1}^1u_n^2\\
u_k^1v_n^2& v_{k+1}^1v_n^2
\end{pmatrix}
\]
and
\[
[\cR(\lambda){\bf w}]_n=\sum\limits_{k=-\infty}^{\infty}
[\cR(\lambda)]_{k,n} \begin{pmatrix}
u_k\\v_{k+1}
\end{pmatrix}.
\]
\end{lemma}
Similarly, for any $\lambda\in \C\setminus\overline\Gamma$, $\re\lambda\le 0$, we obtain
\begin{equation}\label{tRJ-rep}
[\cR(\lambda)]_{k,n}=\frac{1}{\tilde W(\theta(\lambda))}
\begin{cases}
\tilde{\bf w}_k^-(\theta(\lambda))\otimes\tilde{\bf w}_n^+(\theta(\lambda)),\quad k\le n,\\
\tilde{\bf w}_k^+(\theta(\lambda))\otimes\tilde{\bf w}_n^-(\theta(\lambda)),\quad k\ge n,
\end{cases}
\end{equation}
where $\tilde W(\theta)=W(\tilde{\bf w}^+(\theta),\tilde{\bf w}^-(\theta))$.

The representations \eqref{RJ-rep}--\eqref{tRJ-rep}, 
the fact that $W(\theta)$ and $\tilde W(\theta)$ do not vanish for $\lambda\in\Gamma$, 
and the bound \eqref{De0} imply the limiting absorption principle for the perturbed one-dimensional Dirac equation.
\begin{lemma}\label{BV}
Let $q^{ij}\in\ell^1$. Then the convergence
\begin{equation}\label{esk}
    \cR(\lambda\pm \I\varepsilon)\to \cR(\lambda\pm \I0),\quad\varepsilon\to 0+,\quad \lambda\in\Gamma
\end{equation}
holds in ${\mathcal L}(\bl^2_\sigma,\bl^2_{-\sigma})$ with $\sigma>1/2$.
\end{lemma}
\begin{proof}
For any $\lambda\in\Gamma$ and any $n,k\in\Z$, there exist the pointwise limit
\[
[\cR(\lambda\pm \I\ve)]_{n,k}\to[\cR(\lambda\pm \I 0)]_{n,k},\quad \ve\to 0.
\]
Moreover, the bound \eqref{De0} implies that $|[\cR(\lambda\pm \I\ve)]_{n,k}|\le C(\lambda)$.
Hence, the Hilbert--Schmidt norm of the difference $\cR(\lambda\pm \I\ve)-\cR(\lambda\pm \I 0)$
converges to zero in ${\mathcal L}(\bl^2_\sigma,\bl^2_{-\sigma})$ with $\sigma>1/2$ by the dominated convergence theorem.
\end{proof}

Of course the limiting absorption principle implies that the spectrum of $D$ is purely absolutely continuous on $\Gamma$.
\begin{corollary}\label{J-rep}
For any fixed $\sigma>1/2$, the operators
$\cR^{\pm}(\lambda):=\cR(\lambda\pm i0):\bl^2_\sigma\to \bl^2_{-\sigma}$ 
have integral kernels given by
\begin{equation}\label{RJ1-rep}
[\cR^{\pm}(\lambda)]_{n,k} = \frac{1}{W(\theta_{\pm})}\begin{cases}
{\bf w}_n^+(\theta_{\pm})\otimes {\bf w}_k^-(\theta_{\pm}), & n \ge k, \\
{\bf w}_k^+(\theta_{\pm})\otimes {\bf w}_n^-(\theta_{\pm}), & n\le k,
\end{cases}
\end{equation}
for $\lambda\in \Gamma_+=(m,\sqrt{m^2+4})$ and 
\begin{equation}\label{RJ2-rep}
[\cR^{\pm}(\lambda)]_{n,k} = \frac{1}{\tilde W(\theta_{\pm})}\begin{cases}
\tilde{\bf w}_n^+(\theta_{\pm})\otimes 
\tilde{\bf w}_k^-(\theta_{\pm}), & n \ge k, \\
\tilde{\bf w}_k^+(\theta_{\pm})\otimes 
\tilde{\bf w}_n^-(\theta_{\pm}), & n\le k ,
\end{cases}
\end{equation}
for $\lambda\in \Gamma_-=(-\sqrt{m^2+4},-m)$,
where $\theta_+$, and $\theta_-=-\theta_+$
are defined by \eqref{theta_pm}.
\end{corollary}

\section{Dispersive decay}
\label{ll-sec}

Now we are able to prove the $\bl^1\to \bl^\infty$ decay.
\begin {theorem}\label{end1}
Let $q^{ij}\in\ell^1_2$ in the non-resonant case and $q^{ij}\in\ell^1_3$ in the resonant case. 
Then the asymptotics \eqref{fullp}  and \eqref{as1-new} hold, i.e.,
\begin{equation}\label{D-as1}
\Vert \E^{-\I t\D}P_{c}\Vert_{\bl^1\to \bl^\infty}=\mathcal{O}(t^{-1/3}),\quad t\to\infty,
\end{equation}
and
\begin{equation}\label{D-as2}
\Vert \E^{-\I t\D} P_c\Vert_{\bl^2_\sigma\to \bl^2_{-\sigma}}
=\mathcal{O}(t^{-1/2}),\quad t\to\infty,\quad\sigma>1/2.
\end{equation}
\end{theorem}
\begin{proof}
We apply the spectral representation
\begin{align}\nonumber
   \E^{-\I t\D}P_{c}= \E^{-\I t\D}P_{c}^++ \E^{-\I t\D}P_{c}^-
   =& \frac 1{2\pi \I}\int\limits_{\Gamma_+}
   \E^{-\I t\omega}( \cR(\lambda+\I 0)- \cR(\lambda-\I 0))\,d\lambda\\
   \label{sp-rep-sol}
&+\frac 1{2\pi \I}\int\limits_{\Gamma_-}
   \E^{-\I t\omega}( \cR(\lambda+\I 0)- \cR(\lambda-\I 0))\,d\lambda.
\end{align}
We consider the first summand only.
Expressing the kernel of the resolvent in terms of the Jost solutions, 
the kernel of $\E^{-\I t\D} P_{c}^+$ reads:
\begin{align*}
\left[ \E^{-\I t\D}P_{c}^+ \right]_{n,k}& =
\frac{1}{2 \pi\I}\int_{\Gamma_+}\E^{-\I t\lambda}
\left[ \frac{{\bf w}_k^+(\theta_+)\otimes {\bf w}_n^-(\theta_+)}{W(\theta_+)}
- \frac{{\bf w}_k^+(\theta_-)\otimes {\bf w}_n^-(\theta_-)}{W(\theta_-)} \right] d\lambda 
\\
& =  \frac{1}{2\pi \I} \int_{-\pi}^{\pi} \frac{\E^{-\I tg(\theta)}}{g(\theta)}\,
\frac{{\bf w}_k^+(\theta)\otimes {\bf w}_n^-(\theta)}{W(\theta)}  \sin\theta\, d\theta
\end{align*}
for $n\le k$ and by symmetry 
$\left[ \E^{-\I t\D}P_{c}^+\right]_{n,k}= \left[ \E^{-\I t\D}P_{c}^+ \right]_{k,n}$ for $n\ge k$.
\smallskip\\
{\it Step i)}
For \eqref{D-as1} it suffices to prove that
\begin{equation}
\label{eq:m21}
\sup_{n,k\in\Z}|\left[ \E^{\I tH}P_{c}^+\right]_{n,k}| =\mathcal{O}( t^{-1/3}),\quad t\to\infty.
\end{equation}
independent of $n,k$. We suppose $n\le k$ for notational simplicity.
Then
\begin{equation}\label{EDP-rep}
\left[ \E^{-\I t\D}P_{c}^+ \right]_{n,k} =  -\frac{m+\lambda}{4\pi} \int_{-\pi}^{\pi}
\frac{\E^{-\I t\Phi_v(\theta)}}{g(\theta)}\,T(\theta){\bf h}_k^+(\theta)\otimes{\bf h}_n^-(\theta) d\theta,
\end{equation}
where $v=\frac{k-n}{t}\ge 0$, $\Phi_v(\theta)$ is defined in \eqref{phi}, and
\[
T(\theta)=\frac{1}{a_-(\theta)}=\frac{2i\sin\theta}{(m+\lambda)W(\theta)}. 
\]
The quantities $T(\theta)$ and
\[ 
R^{\pm}(\theta)=\frac{b_{\pm}(\theta)}{a_{\pm}(\theta)}=\pm\frac{W^{\pm}(\theta)}{W(\theta)}
\]
are known as transmission and reflection coefficient, respectively.
From \eqref{scatr} there follow the scattering relations
\begin{equation}\label{scat-rel}
T(\theta){\bf w}^{\pm}(\theta)= R^{\mp}(\theta){\bf w}^{\mp}(\theta)+{\bf w}^{\mp}(-\theta),
\quad\theta\in[-\pi,\pi],
\end{equation}
which imply the representations
\[
T(\theta){\bf h}_k^+(\theta)\otimes{\bf h}_n^-(\theta)=\!\left\{\!\!\!\begin{array}{ll}
R^-(\theta){\bf h}_n^-(\theta)\otimes
{\bf h}_k^-(\theta)\E^{2\I k\theta} + {\bf h}_n^-(\theta)\otimes{\bf h}_k^-(-\theta), &\! n\leq k\leq 0,\\[1.5mm]                             
R^+(\theta){\bf h}_k^+(\theta)\otimes
{\bf h}_n^+(\theta)\E^{-2\I n\theta} + {\bf h}_k^+(\theta)\otimes{\bf h}_n^+(-\theta), &\! 0\leq n\leq k.
\end{array}\right.
\]
Using  the identities
\begin{align}\nonumber
\I(n-k)+2\I k&=\I(k+n)=-\I|k+n|,\quad n\le k\le 0,\\
\label{nk}
\I(n-k)-2\I n&=-\I(k+n)=-\I|k+n|,\quad 0\le n\le k,
\end{align}
we rewrite \eqref{EDP-rep}  as
\begin{equation}\label{EDP-rep1}
\left[ \E^{-\I t\D}P_{c}^+ \right]_{n,k}=-\frac{m+\lambda}{4\pi} \Big[\int_{-\pi}^{\pi}
\frac{\E^{-\I t\Phi_v(\theta)}}{g(\theta)}\,Y_{n,k}^{1}(\theta)d\theta
+\!\int_{-\pi}^{\pi}
\frac{\E^{-\I t\tilde\Phi_v(\theta)}}{g(\theta)}\,Y_{n,k}^{2}(\theta)d\theta\Big],
\end{equation}
where 
\begin{align*}
Y_{n,k}^1(\theta)&= \begin{cases}
 T(\theta){\bf h}_k^+(\theta)\otimes{\bf h}_n^-(\theta),~~n\leq 0\leq k,\\
{\bf h}_n^-(\theta)\otimes {\bf h}_k^-(-\theta),~~n\leq k\leq 0,\\
{\bf h}_k^+(\theta)\otimes {\bf h}_n^+(-\theta),~~0\leq n\leq k,
\end{cases}\\
Y_{n,k}^2(\theta)&= \begin{cases}
0,~~n\le 0\le k,\\
R^-(\theta){\bf h}_n^-(\theta)\otimes{\bf h}_k^-(\theta),~~n\leq k\leq 0,\\
R^+(\theta){\bf h}_n^+(\theta)\otimes{\bf h}_k^+(\theta),~~0\leq n\leq k,
\end{cases}
\end{align*} 
and
\begin{equation}\label{tphi}
\tilde\Phi_v(\theta)=g(\theta)+\tilde v\theta,~~{\rm with}~~ \tilde v=|n+k|/t\ge 0.
\end{equation}
In the non-resonant case and $q^{ij}\in \ell^1_2$, we have
\begin{equation}\label{Y-est}
|Y_{n,k}^j(\theta)|+|\frac{\partial}{\partial \theta}Y_{n,k}^j(\theta)|\le C,~~j=1,2,
\quad n\le k,
\end{equation} 
by Proposition \ref{Jost-sol} (ii).
It the resonant case for \eqref{Y-est}, we need $q^{ij}\in \ell^1_3$.
Then 
$W(\theta)$ and $W^{\pm}(\theta)$ can be differentiated two times on $[-\pi,\pi]$ 
according to Proposition \ref{Jost-sol} (ii), and hence
\[
|\frac{\partial^p}{\partial \theta^p}T(\theta)|,\,
|\frac{\partial^p}{\partial \theta^p}R^{\pm}(\theta)|\le C, \quad p=0.1,\quad\theta\in [-\pi,\pi],
\]
due to Remark \ref{W/W}. Therefore \eqref{Y-est} follows by Proposition \ref{Jost-sol} (ii).

Now, as in the proof of Proposition~\ref{free-as} (i), we split the domains of integration 
in \eqref{EDP-rep1} into regions where
either the second or third derivative of the phases is nonzero and
apply the van der Corput Lemma~\ref{lem:vC0} to obtain \eqref{eq:m21}.
\smallskip\\
{\it Step ii)}
We represent $\E^{-\I t\D}P_{c}^+$ as the sum
\[
\E^{-\I t\D}P_{c}^+ ={\mathcal K}^{\pm}(t)+{\mathcal K}(t),
\]
where
\begin{align*}
[{\mathcal K}^\pm(t)]_{n,k}&=-\frac{1}{4\pi} \int_{{\bf J}_{\pm}}
\Big[\E^{-\I t\Phi_v(\theta)} Y_{n,k}^1(\theta) +
\E^{-\I \tilde\Phi_v(\theta)} Y_{n,k}^2(\theta)\Big] \frac{d\theta}{g(\theta)},\\
[{\mathcal K}(t)]_{n,k}&=-\frac{1}{4\pi} \int_{\bf J}
\Big[\E^{-\I t\Phi_v(\theta)} Y_{n,k}^1(\theta)+\E^{-\I \tilde\Phi_v(\theta)} Y_{n,k}^2(\theta)\Big] 
\frac{d\theta}{g(\theta)},
\end{align*}
where ${\bf J}_{\pm}$ and $\bf J$ are defined in  \eqref{Jpm}.
The the van der Corput Lemma~\ref{lem:vC0} with $k=2$  and the estimate \eqref{Y-est} imply
\[
\sup_{n\le k}|[{\mathcal K}(t)]_{n,k}|\le C t^{-1/2},\quad t\ge 1.
\]
Hence
\[
\Vert {\mathcal K}(t)\Vert_{\bl^2_\sigma\to\bl^2_{-\sigma}}\le C t^{-1/2},
\quad\sigma>1/2,\quad t\ge 1.
\]
Furthermore, integration by parts gives
\[
\sup_{n\le k}|[{\mathcal K}^-(t)]_{n,k}|\le C t^{-1},\quad t\ge 1,
\]
and hence
\[
\Vert {\mathcal K}^-(t)\Vert_{\bl^2_\sigma\to\bl^2_{-\sigma}}\le C t^{-1},
\quad\sigma>1/2,\quad t\ge 1.
\]
To estimate ${\mathcal K}^+(t)$ we employ the following lemma. 
\begin{lemma}\label{lem:+}
Let  $Y_{n,k}(\theta)$ satisfies 
\begin{equation}\label{newZ-est}
|Y_{n,k}(\theta)|+|\frac{d}{d\theta}Y_{n,k}(\theta)|
\le (1+|n|^{p})(1 + |k|^{p}),\quad n,k\in\Z,
\end{equation}
with some $p\ge 0$. Let $\Phi_{v}(\theta)=g(\theta)+v\theta$,
where $v=|k- n|/t$ or $v=|k+n|/t$. Then for the operator $K(t)$ with the kernel
\[
[K(t)]_{n,k}=\int_{{\bf J}_+}
\E^{-\I t \Phi_v(\theta)}Y_{n,k}(\theta)d\theta
\]
the following asymptotics hold
\begin{equation}\label{newK-as}
\Vert K(t)\Vert_{\ell^2_{p+\sigma}\to\ell^2_{-(p+\sigma)}}={\mathcal O}(t^{-1/2}), \quad t\to\infty, \quad \sigma>1/2.
\end{equation}
\end{lemma}
\begin{proof}
In the case $v=|k- n|/t$,
we repeat literally the main estimates of Proposition \ref{free-as} (ii).
Namely, consider the operators $K_j(t)$ which differ from the operators defined in \eqref{Kj-def}
by the additional factor $Y_{n,k}(\theta)$.  Then  \eqref{C7}  will change according to 
\begin{equation}\label{C7-new}
|[K_j(t)]_{n,k}|\le C t^{-1/2}t_j^{-1/2}(1+|n|^{p})(1 + |k|^{p})
\end{equation}
by virtue of the van der Corput~Lemma~\ref{lem:vC0}.  Furthermore, instead of \eqref{GG}, 
we will use the following estimate
\begin{align*}
\sum_{n,k\in T_j}\frac {([K_j(t)]_{n,k})^2}{(1+|n|)^{2\sigma+2p}(1+|k|)^{2\sigma+2p}}
&\leq\sup_{n,k\in \Z}(\frac{([K_j(t)]_{n,k})^2}{(1+|n|)^{2p}(1+|k|)^{2p}} b_j(t)\\
&\leq C t^{-2\sigma + \ve}\leq C t^{-1}, \quad j=1,\dots,N,
\end{align*}
which implies \eqref{newK-as} in the first case corresponding to $|v_0-v|\le \frac 12 v_0t_jt^\ve$.
Moreover, \eqref{newZ-est}  gives the additional factor $(1+|n|^{p})(1 + |k|^{p})$ in the right-hand side of \eqref{denom} and \eqref{denom1}  
and hence \eqref{newK-as} in the second case when $|v_0-v|\ge \frac 12 v_0t_jt^\ve$.
In the case $v=|k+ n|/t$ the proof is similar.
\end{proof}
Thus, applying \eqref{Y-est} and the Lemma \ref{lem:+} with $p=0$,  we obtain
\[
\Vert {\mathcal K}^+(t)\Vert_{\bl^2_\sigma\to\bl^2_{-\sigma}}\le C t^{-1/2},\quad\sigma>1/2,
\quad t\to\infty.\qedhere
\]
\end{proof}
Now we consider the non-resonant case and prove 
the asymptotics \eqref{as-new} and \eqref{as2-new}.
\begin {theorem}\label{t-new}
Let $q^{ij}\in\ell^1_3$. Then in the non-resonant case 
\begin{equation}\label{Schr-asn1}
\Vert \E^{-\I t\D}P_c\Vert_{\bl^1_{3/2}\to \bl^\infty_{-3/2}}=\mathcal{O}(t^{-4/3}),\quad t\to\infty,
\end{equation}
\begin{equation}\label{Schr-asn2}
\Vert \E^{-\I tH} P_c\Vert_{\ell^2_\sigma\to \ell^2_{-\sigma}}=\mathcal{O}(t^{-3/2}),\quad t\to\infty,\quad\sigma>2.
\end{equation}
\end{theorem}
\begin{proof}
{\it Step i)} To prove \eqref{Schr-asn1} it suffices to show that
\begin{equation}\label{HP-n1}
|\left[ \E^{-\I t \D}P_c \right]_{n,k}|\le C (1+|n|^{3/2})(1+|k|^{3/2})t^{-4/3},\quad t\ge 1.
\end{equation}
For $n\le k$, we represent the jump of the resolvent across the spectrum  as
 (cf.\ \cite[p13]{EKT})
\[
\cR(\omega+\I 0)- \cR(\omega-\I 0))=\frac{|T(\theta)|^2(m+\lambda)}{-2\I\sin\theta}
[{\bf w}_k^+(\theta)\otimes{\bf w}_n^+(-\theta)+{\bf w}_k^-(\theta)\otimes{\bf w}_n^-(-\theta)].
\]
Inserting this into \eqref{sp-rep-sol} and integrating by parts, we get
\begin{align}\nonumber
&\left[ \E^{-\I t\D}P_c \right]_{n,k}\\
\nonumber
& = \frac{\I(m+\lambda)}{2\pi t}\int_{-\pi}^{\pi}\E^{-\I t g(\theta)}
\frac{d}{d\theta}\Big[\frac{|T(\theta)|^2}{\sin\theta}
({\bf w}_k^+(\theta)\otimes{\bf w}_n^+(-\theta)+{\bf w}_k^-(\theta)\otimes{\bf w}_n^-(-\theta))\Big]d\theta\\
\nonumber
&=[{\mathcal P}^+(t) ]_{n,k}+ [{\mathcal P}^-(t)]_{n,k}.
\end{align}
We consider the first summand only. Evaluating the derivative, we further obtain
\begin{align}\nonumber
\left[{\mathcal P}^+(t)\right]_{n,k}
&=\frac{\I(m+\lambda)}{2\pi t}\int_{-\pi}^{\pi}\E^{-\I t g(\theta)}
\frac{d}{d\theta}\Big[\frac{|T(\theta)|^2}{\sin\theta}\E^{-\I\theta(k-n)}
{\bf h} _k^+(\theta)\otimes{\bf h}_n^+(-\theta)\Big]d\theta\\
\label{I-est}
&=\frac{m+\lambda}{2\pi t}\int_{-\pi}^{\pi}\E^{-\I t \Phi_v(\theta)}
\big((k-n)+\I\frac{d}{d\theta}\big)
\frac{|T(\theta)|^2}{\sin\theta}{\bf h} _k^+(\theta)\otimes{\bf h}_n^+(-\theta),
\end{align}
where $\Phi_v(\theta)$ is defined in \eqref{phi}.
Applying the scattering relation 
\[
T(-\theta){\bf h}_n^+(-\theta)=R^{-}(-\theta){\bf h}_n^-(-\theta)\E^{-2in\theta}+{\bf h}_n^-(\theta),
\]
we get the representation
\begin{equation}\label{I-estn}
\left[{\mathcal P}^+(t) \right]_{n,k}^{+}=
\frac{1}{2\pi t}\int_{-\pi}^{\pi}\E^{-\I t \Phi_v(\theta)}Z_{n,k}^1(\theta)d\theta
+\frac{1}{2\pi t}\int_{-\pi}^{\pi}\E^{-\I t \breve\Phi_v(\theta)}Z_{n,k}^2(\theta)d\theta,
\end{equation}
where $\Phi_v(\theta)$ is the same as before, 
\begin{equation}\label{tphin}
\breve\Phi_v(\theta)=g(\theta)+\breve v\theta,~~{\rm with}~~ \breve v=(n+k)/t,
\end{equation}
and
\begin{align*}
Z_{n,k}^1(\theta)&= \begin{cases}
\big((k-n)+\I\frac{d}{d\theta}\big)
\frac{|T(\theta)|^2}{\sin\theta}{\bf h} _k^+(\theta)\otimes{\bf h}_n^+(-\theta), & 0\leq n\le k,\\[1mm]
\big((k-n)+\I\frac{d}{d\theta}\big)
\frac{T(\theta)}{\sin\theta}{\bf h} _k^+(\theta)\otimes{\bf h}_n^-(\theta),
& n\leq k\leq 0\cup n\le 0\le k,
\end{cases}\\
Z_{n,k}^2(\theta)&= \begin{cases}
0, & 0\leq n\le k,\\[1mm]
\big((k+n)+\I\frac{d}{d\theta}\big)
\frac{T(\theta)}{\sin\theta}R^{-}(-\theta){\bf h} _k^+(\theta)\otimes{\bf h}_n^-(-\theta), & n\leq k\leq 0\cup n\le 0\le k.
\end{cases}
\end{align*}
Since $\frac{T(\theta)}{\sin\theta}=\frac {2\I}{(m+\lambda)W(\theta)}$ then, in the non-resonant case, Proposition~\ref{Jost-sol} (ii) implies
\begin{align}\nonumber
|Z_{n,k}^j(\theta)|+|\frac{d}{d\theta}Z_{n,k}^j(\theta)|&\le C (1+|k|),\quad 0\le n\le k,\\
\nonumber
|Z_{n,k}^j(\theta)|+|\frac{d}{d\theta}Z_{n,k}^j(\theta)|
&\le C (1+\max\{|k|, |n|\}),\quad n\le 0\le k,\\
\nonumber
|Z_{n,k}^j(\theta)|+|\frac{d}{d\theta}Z_{n,k}^j(\theta)|&\le C \big((1+|n|)(1+|k|^2)+(1+|k|^3)\big)\\
\label{Z3-est}
&\le C(1+|n|^{3/2})(1+|k|^{3/2}),\quad n\le k\le 0,
\end{align}
and we obtain \eqref{HP-n1} by the method of Theorem~\ref{end1} (i) .
\smallskip\\
{\it Step ii)}.
To prove \eqref{Schr-asn2},  we represent ${\mathcal P}^+$ as the sum
\begin{equation}\label{MMM}
{\mathcal P}^+(t) ={\bf K}^{\pm}(t)+\breve{{\bf K}}^{\pm}(t)+{\bf K}(t),
\end{equation}
where
\begin{align*}
[{\bf K}^\pm(t)]_{n,k}&= \frac{m+\lambda}{2\pi t}\int_{{\bf J}_{\pm}}
\E^{-\I t\Phi_v(\theta)} Z_{n,k}^1(\theta) d\theta,\\
[\breve{\bf K}^\pm(t)]_{n,k}&= \frac{m+\lambda}{2\pi t}\int_{{\bf J}_{\pm}}
 \E^{-\I t\breve\Phi_v(\theta)}  Z_{n,k}^2(\theta) d\theta,\\
[{\bf K}(t)]_{n,k}&=\frac{m+\lambda}{2\pi t}\int_{\bf J}
\Big(\E^{-\I t\Phi_v(\theta)}Z_{n,k}^1(\theta) +
\E^{-\I t\breve\Phi_v(\theta)} Z_{n,k}^2(\theta)\Big) d\theta,
\end{align*}
where ${\bf J}_{\pm}$ and $\bf J$ are defined in \eqref{Jpm}.
By virtue of \eqref{Z3-est}, the van der Corput Lemma~\ref{lem:vC0} and integration by parts, we obtain
\begin{align*}
|[{\bf K}(t)]_{n,k}|&\le C t^{-3/2}(1+|n|^{3/2})(1 + |k|^{3/2}),\\
|[{\bf K}^-(t)]_{n,k}|&\le C t^{-2}(1+|n|^{3/2})(1 + |k|^{3/2}),\quad t\ge 1,
\end{align*}
respectively.
Then
\[
\Vert {\bf K}(t)\Vert_{\bl^2_\sigma\to\bl^2_{-\sigma}}\le C t^{-3/2},\quad
\Vert {\bf K}^-(t)\Vert_{\bl^2_\sigma\to\bl^2_{-\sigma}}\le C t^{-2},\quad\sigma>2,
\quad t\ge 1.
\]
To  estimate  ${\bf K}^+(t)$,  we apply \eqref{Z3-est} together with  Lemma \ref{lem:+}  with $p=3/2$
and obtain
\[
\Vert {\bf K}^+(t)\Vert_{\bl^2_\sigma\to\bl^2_{-\sigma}}\le C t^{-3/2},\quad\sigma>2,
\quad t\ge 1.
\]
It remains to estimate $\breve{\bf K}^\pm(t)$.  We split $\breve{\bf K}^\pm(t)$ according to
\[
\breve{\bf K}^+(t)=\breve{\bf K}^+_+(t)+\breve{\bf K}^+_-(t),\quad
\breve{\bf K}^-(t)=\breve{\bf K}^-_+(t)+\breve{\bf K}^-_-(t),
\]
where the kernels of the corresponding operators are of the form
\[
[\breve{\bf K}^{\pm}_+(t)]_{n,k}= \begin{cases}
[\breve{\bf K}^{\pm}(t)]_{n,k}, &n+k\ge 0,\\
0, & n+k< 0,
\end{cases}\quad
[\breve{\bf K}^{\pm}_-(t)]_{n,k}= \begin{cases}
0, & n+k\ge0, \\
[\breve{\bf K}^{\pm}(t)]_{n,k},  & n+k<0.
\end{cases}
\]
The operators $\breve{\bf K}^\pm_+(t)$ can be estimated similarly to the operators 
${\bf K}^\pm(t)$. To estimate $\breve{\bf K}^{\pm}_-(t)$ one needs to interchange the 
method for $''+''$ and $''-''$ cases.
\end{proof}

\noindent{\bf Acknowledgments.}
We thank the anonymous referee for the careful reading of our manuscript leading to several
improvements.


\end{document}